\documentclass[11pt,a4paper]{article}
\usepackage{mathrsfs}
\usepackage{amsmath,amsthm,color,eepic}
\usepackage{amssymb} \usepackage{verbatim}

\theoremstyle{theorem}
\newtheorem{theorem}{Theorem}

\newtheorem{lemma}{Lemma}

\theoremstyle{definition}

\newcounter{mathitem}
\newenvironment{mathitem}
  {\begin{list}{{$(\arabic{mathitem})$}}{
   \setcounter{mathitem}{0}
   \usecounter{mathitem}
   \setlength{\topsep}{0pt plus 2pt minus 0pt}
   \setlength{\parskip}{0pt plus 2pt minus 0pt}
   \setlength{\partopsep}{0pt plus 2pt minus 0pt}
   \setlength{\parsep}{0pt plus 2pt minus 0pt}
   \setlength{\leftmargin}{35pt}
   \setlength{\itemindent}{-15pt}
   \setlength{\itemsep}{0pt plus 2pt minus 0pt}}}
  {\end{list}}

\begin{document}

\title{\bf Hamiltonicity of bi-power of bipartite graphs, for finite and infinite cases
\thanks{Supported by NSFC (11601429, 11671320). E-mail: libinlong@mail.nwpu.edu.cn.}}
\date{}

\author{Binlong Li\\[2mm]
\small Department of Applied Mathematics, Northwestern Polytechnical University,\\
\small Xi'an, Shaanxi 710072, P.R. China} \maketitle

\begin{center}
\begin{minipage}{140mm}
\small\noindent{\bf Abstract:} For a graph $G$, the $t$-th power
$G^t$ is the graph on $V(G)$ such that two vertices are adjacent if
and only if they have distance at most $t$ in $G$; and the $t$-th
bi-power $G_B^t$ is the graph on $V(G)$ such that two vertices are
adjacent if and only if their distance in $G$ is odd at most $t$.
Fleischner's theorem states that the square of every 2-connected
finite graph has a Hamiltonian cycle. Georgakopoulos prove that the
square of every 2-connected infinite locally finite graph has a
Hamiltonian circle. In this paper, we consider the Hamiltonicity of
the bi-power of bipartite graphs. We show that for every connected
finite bipartite graph $G$ with a perfect matching, $G_B^3$ has a
Hamiltonian cycle. We also show that if $G$ is a connected infinite
locally finite bipartite graph with a perfect matching, then $G_B^3$
has a Hamiltonian circle.

\smallskip
\noindent{\bf Keywords:} Bipartite graph; infinite graph;
Hamiltonian cycle; Hamiltonian circle
\end{minipage}
\end{center}

%---------------------
\section{Introduction}
%---------------------

A graph $G$ is Hamiltonian if it has a Hamiltonian cycle, i.e., a
cycle containing all vertices of $G$. The $t$-th power $G^t$ of $G$
is the graph on $V(G)$ such that two vertices are adjacent in $G^t$
if and only if they have distance at most $t$ in $G$. The following
classical theorems concern the Hamiltonicity of the power of graphs:

\begin{theorem}[Fleischner \cite{Fl}]\label{ThFl}
If $G$ is a 2-connected finite graph, then $G^2$ is Hamiltonian.
\end{theorem}

\begin{theorem}[Sekanina \cite{Se}]
If $G$ is a connected finite graph of order at least 3, then $G^3$
is Hamiltonian.
\end{theorem}

We consider the analogues of the above theorems on bipartite graphs.
We first notice that a bipartite graph is Hamiltonian only if it is
balanced, i.e., its two bipartition sets have the same size.
Generally, the power of a bipartite graph may not be bipartite. In
order to find a graph operation for bipartite graphs, we pose the
bipartite power (or bi-power for short) of graphs.

For a (bipartite or non-bipartite) graph $G$, we define the $t$-th
bi-power $G_B^t$ as the graph on $V(G)$ with edge set
$$E(G^t_B)=\{xy: d_G(x,y)\mbox{ is odd at most }k\},$$
where $d_G(x,y)$ is the distance between $x$ and $y$ in $G$. Note
that $G^2_B=G$, $G^4_B=G^3_B$, etc. It is nature to ask that is
there $k,t$ such that the $t$-th bi-power of every $k$-connected
balanced bipartite graph is Hamiltonian. The answer is negative by
the following construction.

Let $s\geq t$ be even, and let $V_0,V_1,\ldots,V_{s+1}$ be disjoint
sets of vertices such that $|V_0|=|V_{s+1}|>sk/2$ and $|V_i|=k$ for
$1\leq k\leq s$. Let $G$ be the graph on $\bigcup_{i=0}^{s+1}V_i$ by
adding all possible edges between $V_i$ and $V_{i+1}$, $0\leq i\leq
s$. Thus $G$ is $k$-connected. Since the distance between the
vertices in $V_0$ and in $V_{s+1}$ in $G$ is more that $t$, $V_0\cup
V_{s+1}$ is an independence set of $G_B^t$. It follows that $G_B^t$
is not Hamiltonian.

So we need some additional conditions to get Hamiltonian graphs.

\begin{theorem}\label{ThFinite}
  Let $G$ be a connected finite bipartite graph of order at least 4.
  If $G$ has a perfect matching, then $G^3_B$ is Hamiltonian.
\end{theorem}

We remark that the condition `$G$ has a perfect matching' in Theorem
\ref{ThFinite} cannot be replaced by `$G_B^3$ has a perfect
matching'. Let the bi-star $S_{k,k}$, $k\geq 3$, be the tree with
two vertices of degree $k+1$ and all other vertices of degree 1; and
let $G$ be the graph obtained by subdividing each pendant edge of
$S_{k,k}$ twice. One can check that $G_B^3$ has a perfect matching
but is not Hamiltonian.

Now we turn to the infinite graphs. Thomassen \cite{Th} generalized
Theorem \ref{ThFl} to locally finite graphs with one end. Diestel
\cite{Di05,Di16} launched the ambitious project of extending results
on finite Hamiltonian cycles to Hamiltonian circles in infinite
graphs. Diestel \cite{Di05} then conjectured that the square of any
2-connected locally finite graph has a Hamiltonian circle.
Georgakopoulos \cite{Ge} confirmed Diestel's conjecture. Since it is
necessary to introduce a lot of terminology and notations in order
to state the definition of Hamiltonian circles of infinite graphs,
we refrain from stating the concept explicitly in this introductory
section. Here we present Georgakopoulos's theorem on infinite graphs
concerning our topic. We will explain the concepts in Section 3 (see
also \cite{Di16}, Chapter 8). We apologize for the inconvenience
this may cause.

\begin{theorem}[Georgakopoulos \cite{Ge}]
Suppose that $G$ is an infinite locally finite graph.
\begin{mathitem}
\item If $G$ is 2-connected, then $G^2$ has a Hamiltonian circle.
\item If $G$ is connected, then $G^3$ has a Hamiltonian circle.
\end{mathitem}
\end{theorem}

Several other results in the area of Hamiltonian circles of infinite
graphs can be found in \cite{BrYu,CuWaYu,He15,He16,Le}. Our main
result of the paper is the infinite extension of Theorem
\ref{ThFinite}.

\begin{theorem}\label{ThInfinite}
  Let $G$ be a connected infinite locally finite bipartite graph. If
  $G$ has a perfect matching, then $G^3_B$ has a Hamiltonian circle.
\end{theorem}

The rest of the paper is organized as follows. In Section 2, we
exhibition the proof of Theorem \ref{ThFinite}, with a lemma that
will be also used for our infinite proof. In Section 3, after
introducing the basic terminology and notations, we give some lemmas
and techniques for dealing with the infinite Hamiltonian problems,
following which we complete the proof of Theorem \ref{ThInfinite}.

%----------------------
\section{Finite graphs}
%----------------------

For the purpose of the case of infinite graphs, we first give some
new definitions and a lemma.

Let $G$ be a graph, $A,B\subseteq V(G)$ be disjoint, and
$P=v_0v_1\ldots v_p$ be a nontrivial path of $G$. We say that $P$ is
an \emph{$(A,B)$-path} if $V(P)\cap A=\{v_0\}$ and $V(P)\cap
B=\{v_p\}$; and $P$ is an \emph{$A$-path} if $V(P)\cap
A=\{v_0,v_p\}$. For two disjoint subgraph $H,K$ of $G$, a
$(V(H),V(K))$-path ($V(H)$-path) is also called an $(H,K)$-path
($H$-path). Let $F$ be a path or a cycle, $T$ be a tree of $G$,
$e\in E(T)$ and $T_1,T_2$ be the two components of $T-e$. We say
that $F$ crosses the edge $e$ $k$ times respect to $T$ if $F$
contains $k$ $(T_1,T_2)$-paths.

Now we prove the following lemma.

\begin{lemma}\label{LeFiniteHamiltonian}
If $T$ is a finite tree and $M$ is a perfect matching of $T$, then
for every edge $xy\in M$, $T_B^3$ has a Hamiltonian $(x,y)$-path
crossing every edge $e\in E(T)\backslash M$ exactly twice respect to
$T$.
\end{lemma}

\begin{proof}
We use induction on the order of $T$. The assertion is trivially
true if $T$ has only two vertices. So we assume that $|V(T)|\geq 4$.
Since $T$ is a tree and $xy\in E(T)$, every component of $T-\{x,y\}$
has a neighbor of either $x$ or $y$, but not both. Let
$\mathcal{H}^1=\{H_1^1,\ldots,H_k^1\}$ be the set of components of
$G-\{x,y\}$ that have a neighbor of $x$ and
$\mathcal{H}^2=\{H_1^2,\ldots,H_l^2\}$ be the set of components of
$G-\{x,y\}$ that have a neighbor of $y$. For each
$H_i^1\in\mathcal{H}^1$, let $x_i^1y_i^1\in M$ such that $x$
neighbored $y_i^1$; for each $H_i^2\in\mathcal{H}^2$, let
$x_i^2y_i^2\in M$ such that $y$ neighbored $x_i^2$. By induction
hypothesis, $(H_i^1)_B^3$ has a Hamiltonian $(x_i^1,y_i^1)$-path
$P_i^1$ that crosses every edge $e\in E(H_i^1)\backslash M$ exactly
twice respect to $H_i^1$; and $(H_i^2)_B^3$ has a Hamiltonian
$(x_i^2,y_i^2)$-path $P_i^2$ that crosses every edge $e\in
E(H_i^2)\backslash M$ exactly twice respect to $H_i^2$. If
$\mathcal{H}^1=\emptyset$, then $P=xy_1^2P_1^2x_1^2\ldots
y_l^2P_l^2x_l^2y$ is a Hamiltonian $(x,y)$-path of $G_B^3$ that
crosses every edge $e\in E(T)\backslash M$ exactly twice respect to
$T$. The case of $\mathcal{H}^2=\emptyset$ is similar. If neither
$\mathcal{H}^1$ nor $\mathcal{H}^2$ is empty. then
$P=xy_1^2P_1^2x_1^2\ldots y_l^2P_l^2x_l^2y_1^1P_1^1x_1^1\ldots
y_k^1P_k^1x_k^1y$ is a Hamiltonian $(x,y)$-path of $G_B^3$ that
crosses every edge $e\in E(T)\backslash M$ exactly twice respect to
$T$.
\end{proof}

We say that a balanced bipartite graph $G$ is
\emph{Hamilton-laceable} if for any two vertices $x,y$ in distinct
bipartition sets, $G$ has a Hamiltonian $(x,y)$-path (i.e., an
$(x,y)$-path containing all vertices of $G$). The concept
Hamilton-laceability was introduced by Simmons \cite{Si78,Si81}, and
some times it is called Hamilton-biconnectedness (see
\cite{FaMaMaOr,FaMaOr} for examples). Note that every
Hamilton-laceable balanced bipartite graph (apart from $K_2$) is
Hamiltonian. Now we prove the following result, which is stronger
than Theorem \ref{ThFinite}.

\begin{theorem}
If $G$ is a connected finite bipartite graph that has a perfect
matching, then $G_B^3$ is Hamilton-laceable.
\end{theorem}

\begin{proof}
We use induction on the order of $G$. The assertion is trivial if
$G$ has only two vertices. So we assume that $|V(G)|\geq 4$. If $G$
is not a tree, then it has a spanning tree with a perfect matching,
which can be obtained by taking a perfect matching of $G$ and adding
edges one by one avoiding creating cycles, until no edges can be
added. So we need only consider the case that $G$ is a tree. Let $M$
be a perfect matching of $G$.

Let $x\in X,y\in Y$ be any two vertices, where $X,Y$ are the two
bipartition sets of $G$. We will find a Hamiltonian $(x,y)$-path in
$G_B^3$. Recall that we assume that $G$ is a tree. If $xy\in M$,
then we are done by Lemma \ref{LeFiniteHamiltonian}. So we assume
that $xy\notin M$. It follows that the unique $(x,y)$-path of $G$
contains some edges $e=x'y'\in E(G)\backslash M$, where $x'\in X$
and $y'\in Y$. Thus $G-x'y'$ has exactly two components one of which
contains $x$ and the other contains $y$. Let $H_1,H_2$ be the two
components of $G-x'y'$ containing $x$ and $y$, respectively.

If $y'\in V(H_1)$ and $x'\in V(H_2)$, then by induction hypothesis,
$(H_1)_B^3$ has a Hamiltonian $(x,y')$-path $P_1$ and $(H_2)_B^3$
has a Hamiltonian $(x',y)$-path $P_2$. Thus $P=P_1y'x'P_2$ is an
Hamiltonian $(x,y)$-path of $G_B^3$. If $x'\in V(H_1)$ and $y'\in
V(H_2)$, then let $y'',x''$ be the neighbors of $x',y'$ in $M$,
respectively. By induction hypothesis, $(H_1)_B^3$ has a Hamiltonian
$(x,y'')$-path $P_1$ and $(H_2)_B^3$ has a Hamiltonian
$(x'',y)$-path $P_2$. Note that $x''y''\in E(G_B^3)$, implying that
$P=P_1y''x''P_2$ is a Hamiltonian $(x,y)$-path of $G_B^3$.
\end{proof}

%------------------------
\section{Infinite graphs}
%------------------------

\subsection{Basic terminology and notations}

Now we consider the infinite graphs. We first give the terminology
concerning circles of infinite graphs.

An (infinite) graph $G$ is \emph{locally finite} if every vertex of
$G$ has finite degree. In this section, we always assume that $G$ is
a locally finite graph. A 1-way infinite path is called a
\emph{ray}, and the subrays of a ray are its \emph{tails}. Two rays
of $G$ are \emph{equivalent} if for every finite set $S\subseteq
V(G)$, there is a component of $G-S$ containing tails of both rays.
We write $R_1\cong_GR_2$ if $R_1$ and $R_2$ are equivalent in $G$.
The corresponding equivalence classes of rays are the \emph{ends} of
$G$. We denote by $\varOmega(G)$ the set of ends of $G$.

Let $\alpha\in\varOmega(G)$ and $S\subseteq V(G)$ be a finite set.
We denote by $C_G(S,\alpha)$ the unique component of $G-S$ that
containing a ray (and a tail of every ray) in $\alpha$. We let
$\varOmega_G(S,\alpha)$ be the set of all ends $\beta$ with
$C_G(S,\beta)=C_G(S,\alpha)$. When no confusion occurs, we will
denote $C_G(S,\alpha)$ and $\varOmega_G(S,\alpha)$ by $C(S,\alpha)$
and $\varOmega(S,\alpha)$, respectively.

To built a topological space $|G|$ we associate each edge $uv\in
E(G)$ with a homeomorphic image of the unit real interval $[0,1]$,
where 0,1 map to $u,v$ and different edges may only intersect at
common endpoints. Basic open neighborhoods of points that are
vertices or inner points of edges are defined in the usual way, that
is, in the topology of the 1-complex. For an end $\alpha$ we let the
basic neighborhood
$\widehat{C}(S,\alpha)=C(S,\alpha)\cup\varOmega(S,\alpha)\cup
E(S,\alpha)$, where $E(S,\alpha)$ is the set of all inner points of
the edges between $C(S,\alpha)$ and $S$. This completes the
definition of $|G|$, called the Freudenthal compactification of $G$.
In \cite{Di16} it is shown that if $G$ is connected and locally
finite, then $|G|$ is a compact Hausdorff space.

An \emph{arc} of $G$ a homeomorphic map of the unit interval $[0,1]$
in $|G|$; and a \emph{circle} is a homeomorphic map of the unit
circle $S^1$ in $|G|$. A circle of $G$ is \emph{Hamiltonian} if it
meets every vertex (and then every end) of $G$.

We define a \emph{curve} of $G$ as a continuous map of the unit
interval $[0,1]$ in $|G|$. A curve is \emph{closed} if $0,1$ map to
the same point; and is \emph{Hamiltonian} if it is closed and meets
every vertex of $G$ exactly once. In other words, a Hamiltonian
curve is a continuous map of the unit circle $S^1$ in $|G|$ that
meets every vertex of $G$ exactly once. Note that a Hamiltonian
circle is a Hamiltonian curve but not vice versa.

\subsection{Faithful subgraphs}

For a finite graphs $G$, if $G$ has a spanning subgraph $H$ that is
Hamiltonian, then $G$ itself is Hamiltonian. But this is not true
for infinite graphs, in the meaning that $H$ having a Hamiltonian
circle does not imply $G$ has one. The main reason is that we have
to guarantee injectivity at the ends in Hamiltonian circles. Now we
define a type of subgraphs that are stable on Hamiltonian circles.
We say a subgraph $H$ of $G$ is \emph{faithful} if
\begin{mathitem}
\item every end of $G$ contains a ray of $H$; and
\item for any two rays $R_1,R_2$ of $H$, $R_1\approx_HR_2$ if and only
if $R_1\approx_GR_2$.
\end{mathitem}
If $H\leq G$, then for every finite set $S\in V(H)$, each component
of $H-S$ is contained in a component of $G-S$. Thus the condition
(2) can be replaced by `for any two rays $R_1,R_2$ of $H$,
$R_1\approx_GR_2$ implies $R_1\approx_HR_2$'.

\begin{lemma}\label{FaFaithfulCircle}
Let $H$ be a faithful spanning subgraph of $G$. If $H$ has a
Hamiltonian circle, then $G$ has a Hamiltonian circle.
\end{lemma}

\begin{proof}
We define a map $\pi: \varOmega(H)\rightarrow\varOmega(G)$ such that
for the end $\alpha$ of $H$, $\pi(\alpha)$ is the end of $G$
containing all the rays in $\alpha$. By the definition of the
faithful subgraphs, $\pi$ is a bijection between $\varOmega(H)$ and
$\varOmega(G)$ (see also \cite{Ge}). Let $\alpha\in\varOmega(H)$ and
$S\subseteq V(G)$ be finite. Since $H\leq G$, the component
$C_H(S,\alpha)$ is contained in $C_G(S,\pi(\alpha))$. If there is an
end $\beta\in\varOmega_H(S,\alpha)$, then every ray in $\beta$ has a
tail contained in $C_H(S,\alpha)$, which is contained in
$C_G(S,\pi(\alpha))$. This implies that
$\pi(\beta)\in\varOmega_G(S,\pi(\alpha))$. It follows that
$\pi(\varOmega_H(S,\alpha))\subseteq\varOmega_G(S,\pi(\alpha))$.

Now let $\sigma_H: S^1\rightarrow|H|$ be a Hamiltonian circle of
$H$. We define $\sigma_G: S^1\rightarrow|G|$ such that
$$\sigma_G(p)=\left\{\begin{array}{ll}
  \pi(\sigma_H(p)), & \mbox{if }\sigma_H(p)\in\varOmega(H);\\
  \sigma_H(p),      & \mbox{otherwise}.
\end{array}\right.$$
Clearly the map $\sigma_G$ is injective and meets all vertices of
$V(G)$. Now we prove that it is continuous.

Since $\sigma_H$ is homeomorphic, $\sigma_G$ is continuous at point
$p$ if $\sigma_H(p)$ is a vertex or is an inner point of an edge.
Now we assume that $\sigma_H(p)=\alpha\in\varOmega(H)$. Let
$\boldsymbol{p}=(p_i)_{i=0}^\infty$ be a sequence of points in $S^1$
converges to $p$ and let $S\subseteq V(G)$ be a finite set. Since
$\sigma_H$ is continuous, the neighborhood $\widehat{C}_H(S,\alpha)$
of $\alpha$ contains almost all terms of
$(\sigma(p_i))_{i=0}^\infty$ (that is, there exists $j$ such that
$\sigma(p_i)\in\widehat{C}_H(S,\alpha)$ for all $i\geq j$). Recall
that $C_H(S,\alpha)\subseteq C_G(S,\pi(\alpha))$,
$E_H(S,\alpha)\subseteq E_G(S,\pi(\alpha))$ and
$\pi(\varOmega_H(S,\alpha))\subseteq\varOmega_G(S,\pi(\alpha))$. It
follows that $\widehat{C}_G(S,\pi(\alpha))$ contains almost all
terms of $(\sigma_G(p_i))_{i=0}^\infty$. Thus $\sigma_G$ is
homeomorphic and then is a Hamiltonian circle of $G$.
\end{proof}

\begin{lemma}\label{FaFaithfulFaithful}
Suppose that $K\leq H\leq G$. If $H$ is faithful to $G$ and $K$ is
faithful to $H$, then $K$ is faithful to $G$.
\end{lemma}

\begin{proof}
Let $\alpha_G$ be an arbitrary end of $G$. Since $H$ is faithful to
$G$, $H$ has a ray $R_H\in\alpha_G$. Let $\alpha_H$ be the end of
$H$ with $R_H\in\alpha_H$. Since $K$ is faithful to $H$, $K$ has a
ray $R_K\in\alpha_H$. Thus $R_K\approx_HR_H$, implying that
$R_K\approx_GR_H$, that is, $R_K\in\alpha_G$.

Now let $R_1,R_2$ be two rays of $K$. Since $K$ is faithful to $H$,
$R_1\approx_KR_2$ if and only if $R_1\approx_HR_2$. Since $H$ is
faithful to $G$, $R_1\approx_HR_2$ if and only if $R_1\approx_GR_2$.
This implies that $R_1\approx_KR_2$ if and only if
$R_1\approx_GR_2$. It follows that $K$ is faithful to $G$.
\end{proof}

A \emph{comb} of $G$ is the union of a ray $R$ with infinitely many
disjoint finite paths having precisely their first vertex on $R$;
the last vertices of the paths are the \emph{teeth} of the comb; and
the ray $R$ is the \emph{spine} of the comb. We will use the
following Star-Comb Lemma in our proof.

\begin{lemma}[Diestel, see \cite{Di16}]\label{LeComb}
If $U$ is an infinite set of vertices in a connected graph, then the
graph contains either a comb with all teeth in $U$ or a subdivision
of an infinite star with all leaves in $U$.
\end{lemma}

Since a locally finite graph $G$ contains no infinite stars, Lemma
\ref{LeComb} always yields a comb of $G$. If $H$ is a connected
spanning subgraph of $G$, then for every ray $R$ of $G$, the spine
$R'$ of a comb of $H$ with all teeth in $V(R)$ is a ray in $H$ with
$R\approx_GR'$. Therefore the connected spanning subgraph $H$ is
faithful to $G$ if and only if for any two rays $R_1,R_2$ of $H$,
$R_1\approx_GR_2$ implies $R_1\approx_HR_2$.

\begin{lemma}\label{FaFaithfulSpanning}
Let $H$ be a spanning subgraph of $G$ and $K$ be a spanning subgraph
of $H$. If $K$ is connected and faithful to $G$, then $H$ is
faithful to $G$ and $K$ is faithful to $H$.
\end{lemma}

\begin{proof}
Let $R_1,R_2$ be two rays of $H$ with $R_1\approx_GR_2$. Let
$\alpha_G$ be the end of $G$ containing $R_1,R_2$, let
$R\in\alpha_G$ be a ray of $K$. By Lemma \ref{LeComb}, $K$ has a
comb with all teeth in $V(R_1)$. Let $R'_1$ be the spine of the
comb. Thus $R'_1$ is a ray of $K$ and $R_1\approx_HR'_1$. Since
$H\leq G$, $R_1\approx_GR'_1$ and then $R\approx_GR'_1$. Since $K$
is faithful to $G$, $R\approx_KR'_1$. Since $K\leq H$,
$R\approx_HR'_1$, and then $R\approx_HR_1$. By a similar analysis,
we have $R\approx_HR_2$, and thus $R_1\approx_HR_2$. This implies
that $H$ is faithful to $G$.

Now let $R_1,R_2$ be two rays of $K$ with $R_1\approx_HR_2$. Since
$H$ is faithful to $G$, $R_1\approx_GR_2$. Since $K$ is faithful to
$G$, $R_1\approx_KR_2$. It follows that $K$ is faithful to $H$.
\end{proof}

\begin{lemma}\label{FaFaithfulTree}
Let $\mathcal{P}$ be a partition of $V(G)$ such that $G[P]$ is
connected and finite for every $P\in\mathcal{P}$, let $\mathcal{G}$
be the graph on $\mathcal{P}$ such that for $P_1,P_2\in\mathcal{P}$,
$P_1P_2\in E(\mathcal{G})$ if and only if
$E_G(P_1,P_2)\neq\emptyset$, and let $\mathcal{T}$ be a spanning
tree of $\mathcal{G}$. For every partition set $P\in\mathcal{P}$,
let $T_P$ be a spanning tree of $G[P]$; and for every edge
$f=P_1P_2\in E(\mathcal{G})$, let $e_f$ be an edge in
$E_G(P_1,P_2)$. Let $T$ be the spanning tree of $G$ with edge set
$$\{e_f: f\in E(\mathcal{T})\}\cup\bigcup_{P\in\mathcal{P}}E(T_P).$$
If $\mathcal{T}$ is faithful to $\mathcal{G}$, then $T$ is faithful
to $G$.
\end{lemma}

\begin{proof}
For every ray $R=v_0v_1v_2\ldots$ of $G$, we define a ray $\rho(R)$
of $\mathcal{G}$ as follows: Let $P_0\in\mathcal{P}$ with $v_0\in
P_0$ and $\phi(0)$ be the maximum integer with $v_{\phi(0)}\in R_0$
($\phi(0)$ exists since $P_0$ is finite). Suppose we have already
defined $P_{i-1}$ and $\phi(i-1)$, $i\geq 1$. Let
$P_i\in\mathcal{P}$ such that $v_{\phi(i-1)+1}\in P_i$ and $\phi(i)$
be the maximum integer with $v_{\phi(i)}\in P_i$. Clearly
$v_{\phi(i-1)}\in P_{i-1}$, $v_{\phi(i-1)+1}\in P_i$, implying that
$E_G(P_{i-1},P_i)\neq\emptyset$, and $P_{i-1}P_i\in E(\mathcal{G})$,
$i\geq 1$. Now it follows that $\rho(R)=P_0P_1P_2\ldots$ is a ray of
$\mathcal{G}$. Note that if $R$ is a ray of $T$, then $\rho(R)$ is a
ray of $\mathcal{T}$.

We claim that if $R_1\approx_GR_2$, then
$\rho(R_1)\approx_\mathcal{G}\rho(R_2)$. Let
$\mathcal{S}\subseteq\mathcal{P}=V(\mathcal{G})$ be an arbitrary
finite set. Set $S=\bigcup_{P\in\mathcal{S}}P$. Since each
$P\in\mathcal{P}$ is finite, $S$ is finite. If $R_1\approx_GR_2$,
then there is a component $C$ of $G-S$ that contains a tail $R'_i$
of $R_i$, $i=1,2$. Recall that each $P\in\mathcal{P}$ induces a
connected finite subgraph of $G$. Thus the subgraph $\mathcal{C}$ of
$\mathcal{G}$ induced by $\{P\in\mathcal{P}: P\subseteq V(C)\}$ is a
component of $\mathcal{G}-\mathcal{S}$. Since $\mathcal{C}$ contains
all the vertices $P$ with $P\cap V(R'_i)\neq\emptyset$,
$\mathcal{C}$ contains a tail of $\rho(R_i)$ for $i=1,2$. It follows
that $\rho(R_1)\approx_\mathcal{G}\rho(R_2)$.

For every ray $\mathcal{R}=P_0P_1P_2\ldots$ of $\mathcal{T}$, we
define a ray $\varrho(\mathcal{R})$ of $T$ as follows: Let $u_0\in
P_0$ be a fixed vertex. For $i\geq 0$, let
$e_{P_iP_{i+1}}=v_iu_{i+1}$ be the unique edge in
$E_T(P_i,P_{i+1})$, let $R_i$ be the unique $(u_i,v_i)$-path of
$T_{P_i}$. It follows that
$\varrho(\mathcal{R})=u_0R_0v_0u_1R_1v_1u_2R_2v_2\ldots$ is a ray of
$T$. Note that for every ray $\mathcal{R}$ of $\mathcal{T}$,
$\rho(\varrho(\mathcal{R}))=\mathcal{R}$; and for every ray $R$ of
$T$, $R$ and $\varrho(\rho(R))$ differ only by a finite initial
segments (and so $R\approx_T\varrho(\rho(R))$).

We claim that if $\mathcal{R}_1\approx_\mathcal{T}\mathcal{R}_2$,
then $\varrho(\mathcal{R}_1)\approx_T\varrho(\mathcal{R}_2)$. Let
$S\subseteq V(G)$ be an arbitrary finite set. Set
$\mathcal{S}=\{P\in\mathcal{P}: P\cap S\neq\emptyset\}$. So
$\mathcal{S}$ is a finite subset of $V(\mathcal{G})$. If
$\mathcal{R}_1\approx_\mathcal{T}\mathcal{R}_2$, then there is a
component $\mathcal{C}$ of $\mathcal{T}-\mathcal{S}$ that contains a
tail $\mathcal{R}'_i$ of $\mathcal{R}_i$, $i=1,2$. Let $C$ be the
subgraph of $T$ induced by $\bigcup_{P\in V(\mathcal{C})}P$. It
follows that $C$ is contained in a component of $T-S$. Since $C$
contains all vertices in $\bigcup_{P\in V(\mathcal{R}'_i)}P$, $C$
contains a tail of $\varrho(\mathcal{R}_i)$ for $i=1,2$. It follows
that $\varrho(\mathcal{R}_1)\approx_T\varrho(\mathcal{R}_2)$.

Now we prove the lemma. Suppose that $\mathcal{T}$ is faithful to
$\mathcal{G}$, and $R_1,R_2$ are two rays of $T$ such that
$R_1\approx_GR_2$. It follows that $\rho(R_1),\rho(R_2)$ are two
rays of $\mathcal{T}$ with $\rho(R_1)\approx_\mathcal{G}\rho(R_2)$.
Since $\mathcal{T}$ is faithful to $\mathcal{G}$,
$\rho(R_1)\approx_\mathcal{T}\rho(R_2)$, and thus
$\varrho(\rho(R_1))\approx_T\varrho(\rho(R_2))$. Recall that
$R_1\approx_T\varrho(\rho(R_1))$ and
$R_2\approx_T\varrho(\rho(R_2))$. We have $R_1\approx_TR_2$,
implying that $T$ is faithful to $G$.
\end{proof}

\begin{lemma}\label{FaFaithfulPower}
For any connected graph $G$ and integer $t\geq 1$, $G$ is faithful
to $G^t$ and $G_B^t$.
\end{lemma}

\begin{proof}
Suppose that $R_1,R_2$ are two rays of $G$ with
$R_1\approx_{G^t}R_2$. Let $S\subseteq V(G)$ be an arbitrary finite
set, and set $S'=S\cup N_{G^t}(S)$. Clearly $S'$ is finite, and thus
there is a component $C'$ of $G^t-S'$ that contains tails of both
$R_1$ and $R_2$. For any two adjacent vertices $u,v\in
V(G)\backslash S'$, $G$ has a $(u,v)$-path $P$ of length at most
$t$. Since both $u,v$ have distance more than $t$ from $S$,
$V(P)\cap S=\emptyset$. It follows that $u,v$ are contained in a
common component of $G-S$. This implies that all vertices in $V(C')$
are contained in a common component $C$ of $G-S$. Thus $C$ contains
tails of both $R_1$ and $R_2$, implying that $G$ is faithful to
$G^t$.

Recall that $G_B^t$ is a spanning subgraph of $G^t$. By Lemma
\ref{FaFaithfulSpanning}, $G$ is faithful to $G_B^t$ as well.
\end{proof}

A rooted tree $T$ of $G$ is \emph{normal} if the end-vertices of
every $T$-path in $G$ are comparable in the tree-order of $T$. Note
that if $T$ is spanning, then every $T$-path is an edge of $G$. The
\emph{normal rays} of $T$ are those starting at the root of $T$.
From the following lemma, one can see that a normal spanning tree of
$G$ is faithful to $G$.

\begin{lemma}[Diestel, see \cite{Di16}]\label{LeNormalFaithful}
If $T$ is a normal spanning tree of $G$, then every end of $G$
contains exactly one normal ray of $T$.
\end{lemma}

One can see that the normal spanning tree has a nice property for
the infinite graphs. From the following theorem, we can always find
a normal spanning tree in connected locally finite graphs.

\begin{theorem}[Jung \cite{Ju}]\label{ThNormalTree}
Every countable connected graph has a normal spanning tree.
\end{theorem}

\subsection{Degree of ends}

The \emph{(vertex-)degree} of an end $\alpha\in\varOmega(G)$ is the
maximum number of vertex-disjoint rays in $\alpha$; and the
edge-degree of $\alpha$ is the maximum number of edge-disjoint rays
in $\alpha$. We refer the reader to \cite{BrSt} for some properties
on the end degrees of graphs. Before giving our lemma concerning the
degree of ends, we first list the following K\"onig's Infinite
Lemma.

\begin{lemma}[K\"onig, see \cite{Di16}]\label{LeKonig}
Let $V_0,V_1,V_2,\ldots$ be an infinite sequence of disjoint
non-empty finite sets, and let $G$ be a graph on
$\bigcup_{i=0}^\infty V_i$. Assume that every vertex in $V_i$ has a
neighbor in $V_{i-1}$, $i\geq 1$. Then $G$ has a ray
$R=v_0v_1v_2\ldots$ with $v_i\in V_i$ for all $i\geq 0$.
\end{lemma}

\begin{lemma}\label{FaDegreeArc}
Let $A,B\subseteq V(G)$ be disjoint, and $\alpha\in\varOmega(G)$.
\begin{mathitem}
\item $G$ has $k$ vertex-disjoint $(A,B)$-paths if and only if $|G|$ has
$k$ vertex-disjoint $(A,B)$-curves.
\item $\alpha$ has degree at least $k$ if and only if $|G|$ has $k$
vertex-disjoint nontrivial curves ending in $\alpha$.
\end{mathitem}
\end{lemma}

\begin{proof}
(1) The necessity of the assertion is trivial since a (topological)
path of $G$ is also a curve of $|G|$. Now we prove the sufficiency
of the assertion. Suppose that $G$ has no $k$ vertex-disjoint
$(A,B)$-paths. By Menger's Theorem, there is a set $S\subseteq V(G)$
with $|S|<k$ such that $G-S$ has no $(A,B)$-path. If $|G|$ has $k$
vertex-disjoint $(A,B)$-curves, then one of them is contained in
$|G-S|$. It follows that some component of $G-S$ contains some
vertices of both $A$ and $B$, and thus $G-S$ has an $(A,B)$-path
(see also \cite{DiKu}), a contradiction.

(2) The necessity of the assertion is trivial since a ray in
$\alpha$ is a curve of $|G|$ ending in $\alpha$. Now we prove the
sufficiency of the assertion. Clearly any nontrivial curve ending in
$\alpha$ contains some vertices. For convenience we assume that
$|G|$ has $k$ vertex-disjoint curves between some vertices and
$\alpha$. Let $S_0$ be the set of the starting vertices of the $k$
curves. For $i\geq 1$, set $S_i=S_{i-1}\cup N(S_{i-1})$. Thus $S_i$
is finite and $|G|$ has $k$ vertex-disjoint curves between $S_0$ and
$C(S_i,\alpha)$, for all $i\geq 0$. By (1), $G$ has $k$
vertex-disjoint paths between $S_0$ and $C(S_i,\alpha)$. Let
$\mathcal{V}_i$ be the set of the unions of $k$ vertex-disjoint
paths between $S_0$ and $C(S_i,\alpha)$. Since every path between
$S_0$ and $C(S_i,\alpha)$ is contained in $S_{i+1}$, which is
finite, we can see that $\mathcal{V}_i$ is finite for every $i\geq
0$.

We define a graph $\mathcal{G}$ on
$\bigcup_{i=0}^\infty\mathcal{V}_i$ such that
$U_{i-1}\in\mathcal{V}_{i-1}$ is adjacent to $U_i\in\mathcal{V}_i$
if and only if the $k$ paths of $U_{i-1}$ are the subpaths of the
$k$ paths of $U_i$. Clearly every vertex in $\mathcal{V}_i$ has a
neighbor in $\mathcal{V}_{i-1}$. By Lemma \ref{LeKonig},
$\mathcal{G}$ has a ray $\mathcal{R}=U_0U_1U_2\ldots$ with
$U_i\in\mathcal{V}_i$, $i\geq 0$. It follows that
$\bigcup_{i=0}^\infty U_i$ is the union of $k$ vertex-disjoint rays
in $\alpha$, implying that the degree of $\alpha$ is at least $k$.
\end{proof}

\begin{lemma}\label{FaEndDegreek}
Let $T$ be a faithful spanning tree of $G$ and $F\subseteq E(T)$
such that every component of $T-F$ is finite. If for every edge
$e\in F$, $G$ has at most $k$ edges between the two components
$T_1,T_2$ of $T-e$, then every end of $G$ has degree at most $k$.
\end{lemma}

\begin{proof}
Let $\alpha_G$ be an arbitrary end of $G$, $R$ be a ray of $T$
contained in $\alpha_G$, and $\alpha_T$ be the end of $T$ containing
$R$. We first claim that for every ray $R'\in\alpha_G$ and every
finite subtree $T_0$ of $T$, the component $C=C_T(V(T_0),\alpha_T)$
contains almost all vertices of $R'$. Suppose otherwise that $R'$
has infinitely many vertices contained in $T-C$. Note that $T-C$ is
connected. By Lemma \ref{LeComb}, $T-C$ has a comb with all teeth in
$V(R')$. Let $R''$ be the spine of the comb. It follows that $R''$
is a ray of $T$ and $R'\approx_GR''$. Since $R\approx_GR'$,
$R\approx_GR''$. Since $T$ is faithful to $G$, $R\approx_TR''$,
contradicting the fact that $R''$ has no tail in $C$.

Now we prove the lemma. Let $\alpha_G,\alpha_T$ be defined as above.
Suppose that $\alpha_G$ has degree at least $k+1$. Let $S_0$ be the
starting vertices of $k+1$ vertex-disjoint rays in $\alpha$, $T_0$
be a subtree of $T$ containing $S_0$ and $\mathcal{H}$ be the set of
the components $H$ of $T-F$ with $V(H)\cap V(T_0)\neq\emptyset$. Set
$S_1=\bigcup_{H\in\mathcal{H}}V(H)$, and $T_1=T[S_1]$. Clearly
$S_1\supseteq S_0$ is finite and $T_1$ is a finite subtree of $T$.
Recall that every ray in $\alpha_G$ contains some vertices of
$C_T(S_1,\alpha_T)$. It follows that $G$ has $k+1$ vertex-disjoint
paths between $S_1$ and $C_T(S_1,\alpha_T)$. Let $e$ be the unique
edge of $T$ between $S_1$ and $C_T(S_1,\alpha_T)$. Clearly $e\in F$
and thus $E_G(S_1,C_T(S_1,\alpha_T))\leq k$, a contradiction.
\end{proof}

\subsection{Hamiltonian curves and Hamiltonian circles}

In \cite{KuLiTh}, the authors obtained some necessary and sufficient
conditions for a graph $G$ to have a Hamiltonian curve. We list one
of the conditions which we will use in our paper.

\begin{theorem}[K\"undgen et al.
\cite{KuLiTh}]\label{ThHamiltonianCurve} The graph $G$ has a
Hamiltonian curve if and only if every finite set $S\subseteq V(G)$
is contained in a cycle of $G$.
\end{theorem}

Clearly if a Hamiltonian curve meets every end exactly once, then it
is also a Hamiltonian circle.

\begin{lemma}\label{FaCurveCircle}
If every end of $G$ has degree at most 3, then every Hamiltonian
curve of $G$ is also a Hamiltonian circle.
\end{lemma}

\begin{proof}
It sufficient to show that the Hamiltonian curve passes through each
end exactly once. Suppose that it passes through an end $\alpha$ at
least twice. It is clearly that $G$ has four vertex-disjoint curves
ending in $\alpha$. By Lemma \ref{FaDegreeArc}, $\alpha$ has degree
at least 4, a contradiction.
\end{proof}

\begin{theorem}\label{ThFaithfulHamiltonian}
Let $T$ be a faithful spanning tree of $G$, and $F\subseteq E(T)$
such that every component of $T-F$ is finite. Suppose that for every
subtree $T'$ of $T$, $G$ has a cycle $C'$ such that
\begin{mathitem}
\item $V(T')\subseteq V(C')$, and
\item $C'$ crosses each edges in $F\cap E(T')$ exactly twice respect to $T$.
\end{mathitem}
Then $G$ has a Hamiltonian circle.
\end{theorem}

\begin{proof}
Let $V(G)=\{v_i: i\geq 0\}$. For $i\geq 0$, let $T_i$ be a subtree
of $T$ containing all vertices of $\{v_0,\ldots,v_i\}$, and $C_i$ be
a cycle of $G$ with $V(T_i)\subseteq V(C_i)$ and $C_i$ crosses each
edges in $F\cap E(T_i)$ exactly twice respect to $T$. Set
$\mathcal{C}=(C_i)_{i=0}^\infty$. In the following, we will define a
sequence of infinite subsequences of $\mathcal{C}$, a sequence of
finite subsets of $E(G)$, and a sequence of finite subsets of $F$.

First let $\mathcal{C}^0=(C_i^0)_{i=0}^\infty=\mathcal{C}$ and
$E^0=F^0=\emptyset$. Suppose now we have already defined
$\mathcal{C}^{i-1}$, $E^{i-1}$ and $F^{i-1}$.

Consider the first cycle $C_0^{i-1}$ in $\mathcal{C}^{i-1}$. Let
$\mathcal{H}_i$ be the set of components $H$ of $T-F$ such that
$V(H)\cap V(C_0^{i-1})\neq\emptyset$. Set
$S_i=\bigcup_{H\in\mathcal{H}_i}V(H)$, $S'_i=S_i\cup N(S_i)$,
$E_i=E_G(S_i,G-S_i)$, and $F^i=E_T(S_i,G-S_i)$. Clearly
$F^i\subseteq F$. Since $C_0^{i-1}$ is a cycle and each component of
$T-F$ is finite, we see that $S_i$ is finite. Since $G$ is locally
finite, we have that $S'_i$, $E_i$ and $F^i$ are finite.

Note that there are only finitely many cycles in $\mathcal{C}$ (and
then, in $\mathcal{C}^{i-1}$) that does not contain all vertices in
$S'_i$. It follows that there are infinitely many cycles in
$\mathcal{C}^{i-1}$ that contains all vertices of $S'_i$. Recall
that $E_i$ is finite, and has only finitely many of subsets. So
there is a set $E^i\subseteq E_i$ such that for infinitely many
cycles $C$ in $\mathcal{C}^{i-1}$, $E(C)\cap E_i=E^i$. Let
$\mathcal{C}^i$ be the subsequence of $\mathcal{C}^{i-1}$ consists
of all the cycles $C$ with $S'_i\subseteq V(C)$ and $E(C)\cap
E_i=E^i$.

From the above construction, we can see that every cycle in
$\mathcal{C}^i$ containing all vertices of $C_0^{i-1}$, and for
$i,j\geq 0$,
$$E_i\cap E(C_0^j)=\left\{\begin{array}{ll}
  \emptyset, & \mbox{if }j<i;\\
  E^i, & \mbox{if }j\geq i.
\end{array}\right.$$
Recall that every cycle in $\mathcal{C}^i$ (and then $C_0^i$)
crosses every edge in $F^i$ exactly twice respect to $T$. It follows
that for any edge $e$ of $F^i$, $E^i$ contains exactly two edges in
$E_G(T_1,T_2)$, where $T_1,T_2$ are the two components of $T-e$.

Now let $F'=\bigcup_{i=0}^\infty F^i$ and $G'$ be the spanning
subgraph of $G$ with edge set
$$E(G')=E(T)\cup\bigcup_{i=0}^\infty E(C_0^i).$$ It follows that for
every edge $e\in F'$, $G'$ has at most three edges between the two
components of $T-e$.

We claim that every component of $T-F'$ is finite. Suppose otherwise
that there is an infinite component $H'$ of $T-F'$. Let $v\in
V(H')$. Note that there are only finitely many of cycles in
$\mathcal{C}$ not containing $v$. It follows that there exists $i$
with $v\in V(C_0^{i-1})$. Let $S_i$ be defined as above. Since $S_i$
is finite, there is some edge in $E_T(S_i,G-S_i)\cap E(H')$, which
is contained in $F^i$, a contradiction. Thus we conclude that every
component of $T-F'$ is finite.

By Lemma \ref{FaEndDegreek}, every end of $G'$ has degree at most 3.
Clearly every finite subset of $V(G')$ is contained in a cycle of
$G'$. By Theorem \ref{ThHamiltonianCurve}, $G'$ has a Hamiltonian
curve. By Lemma \ref{FaCurveCircle}, $G'$ has a Hamiltonian circle.
By Lemma \ref{FaFaithfulSpanning}, $G'$ is faithful to $G$. By Lemma
\ref{FaFaithfulCircle}, $G$ has a Hamiltonian circle.
\end{proof}

\subsection{Proof of Theorem \ref{ThInfinite}}

Let $M$ be a perfect matching of $G$. We define a graph
$\mathcal{G}$ on $M$ such that for any two edges $e_1,e_2\in M$,
$e_1e_2\in E(\mathcal{G})$ if and only if $G$ has an edge between
$e_1$ and $e_2$. Clearly $\mathcal{G}$ is connected and locally
finite. By Theorem \ref{ThNormalTree}, $G$ has a normal tree
$\mathcal{T}$, which is faithful to $\mathcal{G}$ by Lemma
\ref{LeNormalFaithful}. By Lemma \ref{FaFaithfulTree}, $G$ has a
faithful spanning tree $T$ containing all edges in $M$. By Lemma
\ref{FaFaithfulPower}, $T$ is faithful to $T_B^3$.

Let $F=E(T)\backslash M$. So every component of $T-F$ consists an
edge in $M$. Let $T'$ be an arbitrary subtree of $T$. By Lemma
\ref{LeFiniteHamiltonian}, $(T')_B^3$ has a Hamiltonian cycle $C'$
that crosses every edge in $F\cap E(T')$ exactly twice respect to
$T'$ (and then respect to $T$ since $C'$ contains no vertices
outsides $T'$). By Theorem \ref{ThFaithfulHamiltonian}, $T_B^3$ has
a Hamiltonian circle.

By Lemmas \ref{FaFaithfulFaithful} and \ref{FaFaithfulPower}, $T$ is
faithful to $G_B^3$. By Lemma \ref{FaFaithfulSpanning}, $T_B^3$ is
faithful to $G_B^3$. By Lemma \ref{FaFaithfulCircle}, $G_B^3$ has a
Hamiltonian circle.

The proof is complete.

\end{document}